\def \version {24nd Mar, 2014}

\documentclass[a4paper,12pt]{article}
\usepackage{array,amsfonts,amsmath,graphicx,mathrsfs,amssymb,amsthm,latexsym}
\usepackage[latin1]{inputenc}
\newtheorem{thm}{Theorem}[section]
\newtheorem{lemma}[thm]{Lemma}

\def \cB {{\cal B}}

\def \cD {{\cal D}}
\def \cE {{\cal E}}
\def \cF {{\cal F}}
\def \cG {{\cal G}}
\def \cH {{\cal H}}

\def \cP {{\cal P}}

\begin{document}
\title{Uniformly resolvable decompositions of $K_v$\\ into paths on two, three and four vertices}

\author {Giovanni Lo Faro
\thanks{Supported  by PRIN, PRA and I.N.D.A.M (G.N.S.A.G.A.), Italy}\\
\small  Dipartimento di Matematica e Informatica \\
\small  Universit\`a di Messina \\
\small  Messina\\
\small Italia\\
{\small \tt lofaro@unime.it}\\
 Salvatore Milici
\thanks{Supported by MIUR and by C. N. R. (G. N. S. A. G. A.), Italy}\\
\small Dipartimento di Matematica e Informatica \\
\small Universit\`a di Catania \\
\small Catania\\
\small Italia\\
{\small \tt milici@dmi.unict.it}  \\
Antoinette Tripodi
\thanks{Supported  by PRIN, PRA and I.N.D.A.M (G.N.S.A.G.A.), Italy}\\
\small  Dipartimento di Matematica e Informatica \\
\small  Universit\`a di Messina \\
\small  Messina\\
\small Italia\\
{\small \tt atripodi@unime.it}}

\date{\small Latest update on \version }
\maketitle

\begin{abstract}
In this paper we consider uniformly resolvable decompositions
of the complete graph $K_v$ into subgraphs such that each
resolution class contains only blocks isomorphic to the same
graph. We completely determine the spectrum for the
case in which all the resolution classes consist of either $P_2$, $P_3$ and $P_4$.
\end{abstract}

\vspace{5 mm}
\noindent AMS classification: $05B05$.\\
Keywords: Resolvable graph decomposition; Uniform resolution;
Paths

\section{Introduction and definitions}\label{introduzione}

Given a collection $\cH$ of graphs, an {\em $\cH$-decomposition\/}
of a graph $G$ is a decomposition of the edge set of $G$ into
subgraphs (called {\em blocks}) isomorphic to some element of $\cH$. Such a decomposition is
said to be {\em resolvable\/} if it is possible to partition the blocks
into  classes $\cP_i$
 (often referred to as {\em parallel classes\/})
 such that every vertex of $G$ appears
in exactly one block of each $\cP_i$. A resolvable $\cH$-decomposition of $G$ is sometimes also
referred to as an {\em $\cH$-fac\-torization\/} of $G$, and a class can
be called an {\em $\cH$-factor\/} of $G$. The case where
 $\cH=\{K_2\}$ (a single edge) is known as a
 {\em 1-factorization\/}; for $G=K_v$
it is well known to exist if and only if $v$ is
even. A single class of a  1-factorization, that is a pairing of
all vertices, is also known as a {\em 1-factor\/}
 or {\em perfect matching}.

In many cases we
wish to place further constraints on the classes. For example, a
class is called {\em uniform} if every block of the class is
isomorphic to the same graph from $\cH$. The
result of Rees \cite{R} which finds necessary and sufficient
conditions for the existence of uniformly resolvable $\{K_2, K_3\}$-decompositions of $K_v$ is of particular note. Uniformly resolvable
decompositions of $K_v$ have also been studied in \cite{URD}, \cite{DQS},
\cite{DLD}, \cite{GM}, \cite{GM1}, \cite{HR}, \cite{KMT}, \cite{KLMT}, \cite{M}, \cite{MT}, \cite{S1}, \cite{S1bis}, \cite{S2}  and \cite{SG}.

If $\cH = \{H_1, H_2, \ldots, H_l\}$, let $(H_1, H_2, \ldots, H_l)$-URD$(v;r_1,r_2, \ldots, r_l)$  denote a uniformly resolvable decomposition of $K_{v}$  into $r_i$ classes containing only copies of the graph $H_i$, for $i=1,2,\ldots,l$. 
In this paper we study the existence of uniformly resolvable decompositions into paths $P_2=K_2$, $P_3$, $P_4$
for the complete graph $K_v$.  The existence of (uniformly) resolvable decompositions for each $\cH \subset\{K_2,P_3,P_4\}$
 was studied separately already long ago:
\begin{itemize}
 \item
There exists a resolvable $K_2$-decomposition of $K_v$ if and only if \ $v \equiv 0\pmod {2}$.
 \item
There exists a resolvable $P_3$-decomposition of $K_v$
 if and only if \ $v \equiv 9\pmod {12}$ \ \cite{H}.
 \item
There exists a resolvable $P_4$-decomposition of $K_v$
 if and only if \ $v \equiv 4\pmod {12}$  \ \cite{BHY}.
\item
There exists a $(K_2, P_3)$-URD$(v;r,s)$  if and only if  $v\equiv 0 \pmod{6}$ and
$(r,s) \in \{(v-1-4x, 3x), x=0,1,\ldots,\frac{v-4}{4}\}$ for $v \equiv0 \pmod{12}$ and $(r,s) \in \{(v-1-4x, 3x), x=0,1,\ldots,\frac{v-2}{4}\}$ for $v \equiv6\pmod{12}$\ \cite{M1,Y}.

\item
There exists a $(K_2, P_4)$-URD$(v;r,s)$  if and only if  $v\equiv 0 \pmod{4}$ and
$(r,s) \in \{(v-1-3x, 2x), x=0,1,\ldots,\frac{v-3}{3}\}$ for $v \equiv0 \pmod{12}$, $(r,s) \in \{(v-1-3x,2x), x=0,1,\ldots,\frac{v-1}{3}\}$ for
$v \equiv4\pmod{12}$ and  $(r,s) \in \{(v-1-3x, 2x), x=0,1,\ldots,\frac{v-2}{3}\}$ for $v\equiv8 \pmod{12}$\ \cite{M1,Y}.
\item
 There exists a$(P_3, P_4)$-URD$(v;r,s)$  if and only if  $v \equiv0 \pmod{12}$ and $(r,s) \in
 \{(6+9x, 2+\frac{2(v-12)}{3}-8x), x=0,1,\ldots,\frac{v-12}{12}\}$ \cite{GM1}.

\end{itemize}

In what follows, we will denote by $[a_1,a_2,\ldots, a_k]$  the path $P_k$, $k\geq 3$, having vertex set $\{a_1,a_2,\ldots, a_k\}$
and edge set $\{\{a_1,a_2\}, \{a_2,a_3\}, \ldots, \{a_{k-1},a_k\}\}$.
Given a graph $G$, we will denote by $G_{(n)}$ the graph on $V(G)\times Z_n$ with edge set $\{\{x_i,y_j\}:\{x,y\}\in \cE (G), i,j\in Z_n\}$. %(when $G=K_m$, $K_{m(n)}$ is the complete multipartite graph with $m$ parts of size $n$;  if $G$ is a cycle of lenght $m$, we will denote  by $C_{m(n)}$).

\section{Necessary conditions}

In this section we will give necessary conditions for the existence
of a $(P_2, P_3, P_4)$-URD$(v;r,s,t)$. To begin with, note that if there exists a $(P_2, P_3, P_4)$-URD$(v;r,s,t)$ with $s>0$ and $t>0$, then $v\equiv 0\pmod{12}$.

\begin {lemma}
\label{lemmaP1}{\rm \cite{M1,Y}}  There exists a $(P_2, P_3, P_4)$-URD$(v;r,0,t)$ if and only if $v\equiv 0\pmod{4}$ and $(r,0,t)\in \{ (v-1-3x,0,2x),  x=0,1,\ldots,\frac{v-3}{3}\}$ for $v \equiv0 \pmod{12}$, $(r,0,t) \{\in(v-1-3x,0,), x=0,1,\ldots,\frac{v-1}{3}\}$ for
$v \equiv4\pmod{12}$ and  $(r,0,t) \in \{(v-1-3x,0,2x), x=0,1,\ldots,\frac{v-2}{3}\}$ for $v\equiv8 \pmod{12}$.
\end {lemma}

\begin{lemma}
\label{lemmaP2}{\rm \cite{M1,Y}} There exists a $(P_2, P_3, P_4)$-URD$(v;r,s,0)$ if and only if\/ $v\equiv 0\pmod{6}$ and \/ $(r,s,0)\in \{(v-1-4x, 3x,0), x=0,1,\ldots,\frac{v-4}{4}\}$ for $v \equiv0 \pmod{12}$ and $(r,s,0) \in \{(v-1-4x, 3x,0), x=0,1,\ldots,\frac{v-2}{4}\}$ for $v \equiv6\pmod{12}$.
\end {lemma}

\begin{lemma}
\label{lemmaP3} Let  $v\equiv 2,10\pmod{12}$. A $(P_2, P_3,P_4)$-URD$(v;r,s,t)$ there  exists if and only if $s=t=0$.
\end{lemma}
\begin{proof}

Suppose  there exists a $(P_2, P_3, P_4)$-URD$(v;r,s,t)$, with $(s,t)\neq(0,0)$. By the resolvability $v$ must be divisible by 4,  6 or  12. A contradiction.
\end{proof}

Given $v\equiv 0\pmod{12}$, for every $0\leq u \leq \frac{v-1}{4}$ define $r(v,u)$   according to the following table:

\vspace{4 mm}

 \begin{minipage}[t]{\textwidth}
\begin{center}
\begin{tabular}{|c|c|}
\hline
% \multicolumn{2}{|c|}
% {Table 2}\\ \hline
  $u  $ &  {$r(v,u)$}
\\
\hline
$0 \pmod{3}$ & $\frac{v-3}{3}-\frac{4u}{3} $\\
$1\pmod{3}$ & $\frac{v-6}{3}-\frac{4(u-1)}{3} $\\
$2 \pmod{3}$ & $\frac{v-9}{3}-\frac{4(u-2)}{3} $\\

\hline
 \end{tabular}

\bigskip

Table 1:  $r(v,u)$

 \end{center}\end{minipage}

\vspace{4 mm}

\noindent and let 
\begin{equation}
 D (v) =\bigcup_{i=0}^{\frac{v-4}{4}} D_{i(v)},
 \end{equation}
where \begin{equation}
 D _{i}(v) =\{(v-1-4i-3y,3i,2y), 0\leq y\leq r(v,i)\},
 \end{equation}
 for $i=0,1,\ldots,\frac{v-4}{4}$.

\begin{lemma}
\label{lemmaP0} Let \ $v\equiv 0\pmod{12}$. If there exists a\/ $(P_2, P_3, P_4)$-URD$(v;r,s,t)$ then $(r,s,t)\in D(v)$.
\end{lemma}

\begin{proof}
Assume that there exists a $(P_2, P_3, P_4)$-URD$(v;r,s,t)$ $\cD$, $s>0,t>0$. By the resolvability,

$$\frac{rv}{2}+\frac{2sv}{3}+\frac{3tv}{4}=\frac{v(v-1)}{2}$$
 and hence
\begin{equation}
  6r+8s+9t=6(v-1). \end{equation}
which implies that \
$s\equiv 0\pmod{3}$ \ and \
$t\equiv 0\pmod{2}$.
Let $s=3x$ and $t=2y$;
the equation $(1)$ gives  $r=v-1-4x-3y$.
Since $r$, $s$ and $t$  cannot be negative, the value of
$x$ and $y$ are in the range as given in the definition of $D(v)$.
\end{proof}

Let now
$URD(v; P_2, P_3, P_4)$ := $\{(r,s,t)$ : $\exists$ $(P_2, P_3, P_4)$-URD$(v;r,s,t)\}$. In this paper we completely solve the spectrum problem for such systems, i.e., characterize the existence of uniformly resolvable
decompositions of $K_{v}$ into $r$  1-factors, $s$ classes containing only copies of $P_3$  and $t$
classes containing only copies of $P_4$ by proving the following result:\\

\noindent \textbf{Main Theorem.} { \em
 For every integer\/ $v\equiv 0\pmod{12}$,   $URD(v;K_ 2, P_3, P_4)$=$D(v)$}.

\section{Costructions and related structures}

In this section we will introduce some useful definitions and results and discuss
constructions we will use in proving the main result. For missing
terms or results that are not explicitly explained in the paper,
the reader is referred to \cite{CD} and its online updates.
For some results below, we also cite this handbook instead of the
 original papers.
 
 \bigskip

An incomplete resolvable $(K_{2}, P_3, P_4)$-decomposition of $K_{v+h}$ with a hole of size $h$ is a
$(K_{2}, P_3, P_4)$-decomposition of $K_{v+h}- K_h$ in which there are two types of
classes, {\em partial} classes which cover every point except
those in the hole (the set of points of $K_h$ are referred to as the {\em hole}) and  {\em full} classes which cover every point of $K_{v+h}$.
Specifically a $(K_{2},P_3, P_4)$-IURD$(v+h, h; [r_1, s_1, t_1], [\bar{r}_1, \bar{s}_1, \bar{t}_1])$ is a uniformly
resolvable $(K_{2}, P_3, P_4)-$decomposition of $K_{v+h}-K_h$ with $r_1$ 1-factors,  $s_1$ classes of copies of $P_3$ and $t_1$ classes of copies of $P_4$, which cover only the points not in the hole, $\bar{r}_1$ 1-factors,  $\bar{s}_1$ classes of copies of $P_3$ and $\bar{t}_1$ classes of copies of $P_4$, which cover every point of $K_{v+h}$.

\bigskip

A (resolvable) $\cH$-decomposition of the complete multipartite
graph with $u$ parts each of size $g$  is known as a (resolvable)
group divisible design $\cH$-(R)GDD of type $g^u$ (the parts of size
$g$ are called the \textit{groups} of the design). When $\cH = K_n$ we will
call it an $n$-(R)GDD. A $(P_2, P_3,P_4)$-URGDD $(r,s,t)$ of type $g^u$ is a uniformly resolvable
decomposition of the complete multipartite graph with $u$ parts
each of size $g$ into $r$ 1-factors, $s$ classes containing only copies of  $P_3$ and $t$ classes containing only copies of $P_4$.
If the blocks of an $\cH$-GDD of type ${g^u}$ can be partitioned into partial
parallel classes, each of which contains all points except those of one group,
we refer to the decomposition as a {\em frame}. When $\cH = K_n$ we will
call it an $n$-{\em frame} and it is easy to deduce that the number of partial parallel classes missing a specified group $G$ is
$\frac{|G|}{n-1}$. We quote the following lemma for a later use.

%\begin{lemma} \label{lemma RF} {\rm (\cite{S2})} There exists a\/ $2$-frame of type  $g^{u}$  if and only if $u\geq 3$, $g(u-1)\equiv 0\pmod{2}$. \end{lemma}

\bigskip

Define $\overline{r}(k,u)$ according to the following table:

\vspace{4 mm}

 \begin{minipage}[t]{\textwidth}
\begin{center}
\begin{tabular}{|c|c|}
\hline
% \multicolumn{2}{|c|}
% {Table 2}\\ \hline
  $u  $ &  {$\overline{r}(k,u)$}
\\
\hline
$0 \pmod{3}$ & $4k-\frac{4u}{3} $\\
$1\pmod{3}$ & $4k-2-\frac{4(u-1)}{3} $\\
$2 \pmod{3}$ & $4k-3-\frac{4(u-2)}{3} $\\

\hline
 \end{tabular}

\bigskip

Table 2: $\overline{r}(v,u)$

 \end{center}\end{minipage}

\vspace{4 mm}

\noindent and let

\begin{equation}
 \bar{D}((12k)^{m}) =\bigcup_{j=0}^{3k(m-1)}\bar{D}_{j}((12k)^{m}),
 \end{equation}
where
\begin{equation}
 \bar{D}_{j}((12k)^{m}) =\{(12k(m-1)-4j-3y,3j,2y), 0\leq y\leq \overline{r}(k(m-1),j)\},\end{equation}
for $j=0,1,\ldots,3k(m-1)$.

\begin{lemma}
\label{lemmaP00} If there exists a\/ $(P_2, P_3, P_4)$-URGDD$(r,s,t)$ of type $12^{k}$, $s>0, t>0$,  then\/ $(r,s,t)\in \bar{D}(12^{k})$.
\end{lemma}

\begin{proof}
Assume  there exists a $(P_2, P_3, P_4)$-URGDD$(r,s,t)$ $\cD$, $s>0,t>0$. By the resolvability of $\cD$ 

$$ 6kr+8ks+9kt=72k(k-1)$$
 and hence
\begin{equation}
  6r+8s+9t=72(k-1). \end{equation}
which  implies that  $(r,s,t)\in \bar{D}(12^{k})$.
\end{proof}
  
 Let $(r_1,s_1, t_1)$ and $(r_2,s_2, t_2)$
be two triples of non-negative integers. Define $(r_1,s_1, t_1) +(r_2,s_2,
t_2)=(r_1+r_2,s_1+s_2, t_1+t_2)$. If $X$ and $Y$ are two sets of triples of
non-negative integers, then $X+Y$ denotes the set $\{(r_1,s_1, t_1)
+(r_2,s_2, t_2) : (r_1,s_1, t_1)\in X, (r_2,s_2, t_2) \in Y \}$. If $X$ is a set
of triples of non-negative integers and $h$ is a positive integer,
then  $h * X$ denotes the set of all triples of non-negative integers
which can be obtained by adding any $h$ elements of $X$ together
(repetitions of elements of $X$ are allowed). To obtain our main result we will use the following lemmas.

\begin{lemma}
\label{lemmaC1} For every $h\geq1$, $h*\bar{D}(12^{2})=\bar{D}((12h)^{2})$.

\end{lemma}

\begin{proof}
We induct on $h$. The case $h=1$ is trivially true. Suppose the assertion holds for $h>1$ and prove it for $h+1$.
We have
$$ (h+1)\ast \bar{D}(12^{2})=h\ast \bar{D}(12^{2})+\bar{D}(12^{2}).$$
By induction hypothesis 
$$ h\ast \bar{D}(12^{2})=\bar{D}((12h)^{2})=\bar{D}_0((12h)^{2}) \cup \bar{D}_1((12h)^{2})\cup \ldots\cup \bar{D}_{3h}((12h)^{2}).$$
Since
$$ \bar{D}(12^{2})=\bar{D}_0(12^{2}) \cup \bar{D}_1(12^{2})\cup \bar{D}_2(12^{2})\cup \bar{D}_{3}(12^{2}),$$
it is easy to chek that

\bigskip

$ h\ast \bar{D}(12^{2})+\bar{D}(12^{2})=$
\bigskip

$=\left\{\bigcup_{j=0}^{3h}\left[\bar{D}_{j}\left(\left(12h\right)^{2}\right)+\bar{D}_0\left(12^{2}\right)\right]
\right\}\bigcup\left\{\bigcup_{i=0}^{3}\left[\bar{D}_{i}\left(12^{2}\right)+\bar{D}_{3h}\left(\left(12h\right)^{2}\right)
\right]\right\}.$

\bigskip

\noindent Now, as  it is easy to see, for each $j=0,1,\ldots,3h$
$$ \bar{D}_{j}((12h)^{2})+\bar{D}_0(12^{2})=\bar{D}_{j}((12h+12)^{2}) $$
and, for each $i=0,1,2,3$,
$$ \bar{D}_{i}(12^{2})+\bar{D}_{3h}((12h)^{2})=\bar{D}_{i+3h}((12h+12)^{2}), $$
and so we obtain

\bigskip

$\ \ \ \ \  h\ast \bar{D}(12^{2})+\bar{D}(12^{2})=\bigcup_{j=0}^{3+3h}\bar{D}_{j}((12h+12)^{2})=\bar{D}((12(h+1))^{2}).$
%This prove the Lemma.
\end{proof}

\begin{lemma}
\label{lemmaF1} For every $h\geq1$, $h*\bar{D}(12^{2})=\bar{D}(12^{h+1})$.
\end{lemma}

\begin{proof}
%Lemma \ref{lemmaC1} gives $h\ast \bar{D}(12^{2})=\bar{D}((12h)^{2})$. 
It is easy to check that $\bar{D}((12h)^{2})=\bar{D}(12^{h+1})$ and so the assertion holds by Lemma \ref{lemmaC1}.\end{proof}

\begin{thm}
\label{thmR2} Let\/ $v$, $g$, $t$, $k$  and\/ $u$   be non-negative
integers such that\/ $v=gtu$. If there exists
\begin{itemize}
\item
[$(1)$] a\/ $k$-RGDD of type\/ $g^{u}$;

\item
[$(2)$] a $(P_2,P_3, P_4)$-URGDD$(r_1,s_1,t_1)$ of\/
type $t^{k}$ with\/ $(r_1, s_1,t_1)\in J_1$;

\item
[$(3)$] a $(P_2, P_3, P_4)$-URD$(gt;r_2,s_2,t_2)$,\/ with\/ $(r_2, s_2,t_2)\in J_2$;
\end{itemize}
then there exists a\/ $(P_2, P_3, P_4)$-URD$(v;r,s,t)$  for
each\/ $(r,s,t)\in J_2+ h\ast J_1$, where\/ $h=\frac{g(u-1)}{2}$ is the
number of parallel classes of the\/ $k$-RGDD of type\/ $g^{u}$.
\end{thm}
\begin{proof}
Let $\cG$ be a $k$-RGDD of type $g^{u}$,
with $u$ groups $G_i$, $i=1,2,\ldots ,u$, of size $g$; let $R_1,R_2,\ldots,R_{\frac{g(u-1)}{k-1}}$,  be the
parallel classes of this $k$-RGDD. Expand each point $t$ times and for each block $b$ of a given resolution class
of $\cG$ place on $b\times\{1,2,\ldots,t\}$ a copy of a $(P_2, P_3, P_4)$-URGDD$(r_1,s_1,t_1)$  of\/
type $t^{k}$ with $(r_1, s_1,t_1)\in J_1$. For each
$i=1,2,\ldots,u$, place on $G_i\times\{1,2,\ldots,t\}$ a copy of a
$(P_2, P_3, P_4)$-URD$(gt;r_2,s_2,t_2)$
with\/ $(r_2, s_2,t_2)\in J_2$. The result is a $(P_2, P_3, P_4)$-URD$(v;r,s,t)$ with $(r,s,t)\in \{J_2+ ({\frac{g(u-1)}{k-1}})\ast
J_1$\}.
\end{proof}

\begin{thm}
\label{thmR3} Let $v$, $g$, $t$, $h$ and $u$  be non-negative integers
such that $v=gtu+h$. If there exists
\begin{itemize}
\item
[$(1)$] a $2$-frame $\cF$ of type $g^{u}$;

\item
[$(2)$] a $(K_{2}, P_3, P_4)$-URD$(h;r_1,s_1,t_1)$  with $(r_1,
s_1,t_1)\in J_1$;

\item
[$(3)$] a $(K_{2}, P_3, P_4)$-URGDD$(r_2,s_2, t_2)$ of type $t^{2}$ with $(r_2,
s_2,t_2)\in J_2$;

\item
[$(4)$] a $(K_{2}, P_3, P_4)$-IURD$(gt+h,h; [r_1, s_1,t_1], [r_3,s_3,t_3])$  with
$(r_1, s_1,t_1)\in J_1 $ and $(r_3,s_3,t_3)\in J_3$= $g\ast J_2$;
\end{itemize}
then exists a $(K_{2}, P_3, P_4)$-URD$(v+h;r,s,t)$  for each
$(r,s,t)\in J_1+ u\ast J_3$.
\end{thm}

\begin{proof}

Let $\cF$ be a $2$-frame of type $g^{u}$ with groups $G_i$, $i=1,2,\ldots,u$; expand each point $t$ times and add a set $H=\{a_1,a_2,\ldots,a_h\}$.
For $j=1,2,\ldots,g$,  let $p_{i,j}$ be the $j$-th partial parallel class which miss the group $G_i$;  for each $b\in p_{i,j}$, place on $b\times\{1,2,\ldots,t\}$ a copy $D_{i,j}^b$  of a $(K_{2}, P_3, P_4)$-URGDD$(r_2,s_2,t_2)$ of type $t^2$,  with $(r_2,s_2,t_2)\in J_2$;   place on $H\cup (G_i\times\{1,2,\ldots,t\})$ a copy $D_i$ of a $(K_{2}, P_3,P_4)$-IURD$(gt+h,h;[r_1, s_1,t_1], [r_3,s_3,t_3])$   with $H$ as hole, $(r_1,s_1,t_1)\in J_1$ and $(r_3,s_3,t_3)\in J_3$= $g\ast J_2$. Now  combine all together the  parallel classes of $D_{i,j}^b$, $b\in p_{i,j}$,
 along with the full classes of   $D_i$ so to obtain $r_3$ 1-factors,   
 $s_3$ classes of paths $P_3$ and $t_3$ classes of paths $P_4$, $(r_3,s_3,t_3)\in J_3$, on $ H\cup (\cup_{i=1}^{u}G_i\times\{1,2,\ldots,t\})$.
Fill the hole $H$ with a copy $D$ of
$(K_{2}, P_3,P_4)$-URD$(h;r_1,s_1,t_1)$  with $(r_1,s_1,t_1)\in J_1$ and
combine the classes of $D$ with the partial classes of $D_i$  so to obtain
$r_1$ 1-factors, $s_1$ classes of paths $P_3$ and $t_1$ classes of paths $P_4$ on $ H\cup (\cup_{i=1}^{u}G_i\times\{1,2,\ldots,t\})$.
The result is a $(K_{2}, P_3, P_4)$-URD$(v+h;r,s,t)$  for each $(r,s,t)\in J_1+ u\ast J_3$. \end{proof}

% To obtain our main result, we need further definitions and results.

\section{Small cases}

\begin{lemma}
\label{lemmaA1} There exists a \/ $(P_2, P_3, P_{4})$-URGDD$(r,s,t)$ of type\/ $6^{2}$,
for every $(r,s,t)\in \{(6,0,0), (3,0,2), (0,0,4), (2,3,0)\}$.
\end{lemma}

\begin{proof}
The cases $(6,0,0), (3,0,2), (0,0,4)$ correspond to a $(P_2, P_4)$-URGDD$(r_1,s_1)$ of type$6^2$, with $(r_1,s_1)\in \{(6,0), (3,2), (0,4)\}$, which is known to exist \cite{M1}.
 For the case $(2,3,0)$ take the groups to be $\{1,2,3,4,5,6\}, \{a,b,c,d,e,f\}$
and the classes listed below:\\
$\{\{1,e\}, \{2,f\}, \{3,c\}, \{4,d\}, \{5,b\}, \{6,a\}\}$, $\{\{1,d\}, \{2,b\}, \{3,a\}, \{4,e\}, \{5,f\}, \{6,c\}\}$,\\
$\{[a,1,b], [f,4,c], [3,d,5],[2,e,6] \}$, $\{[c,2,d], [a,5,e], [3,b,4],[1,f,6]\}$, $\{[e,3,f], [b,6,d],$ \\$[2,a,4],[1,c,5] \}$.

\end{proof}

\begin{lemma}
\label{lemmaA11} There exists a \/ $(P_2, P_3, P_{4})$-URGDD$(r,s,t)$ of type\/ $4^{3}$,
for every $(r,s,t)\in \{(0,6,0), (4,3,0),(1,3,2)\}$.
\end{lemma}

\begin{proof}
The case $(0,6,0)$ follows by \cite{UT}. For the remaining cases, take the groups to be $\{1,2,3,4\}, \{5,6,7,8\}, \{9,10,11,12\}$
and the classes listed below:
\begin{itemize}
%\item $(0,6,0)$: \\ $\{[8,1,11],[7,2,12],[4,5,9], [3,6,10]\},\{ [11,3,5],  [12,4,6], [1,7,9], [2,8,10]\}$,\\ $\{[6,9,8],[5,10,7],[2,11,4], [1,12,3]\}$,\\ $\{[5,1,10],[6,2,9],[4,8,12], [3,7,11]\},\{ [8,3,10],  [7,4,9], [1,6,12], [2,5,11]\}$,\\ $\{[6,11,8],[5,12,7],[2,10,4], [1,9,3]\}$.
\item
$(4,3,0)$:\\
$\{\{1,6\}, \{2,5\}, \{3,10\}, \{4,9\}, \{7,12\}, \{8,11\}\}$, \\
$\{\{1,5\}, \{2,6\}, \{3,9\}, \{4,10\}, \{7,11\}, \{8,12\}\}$,\\
$\{\{1,9\}, \{2,10\}, \{3,7\}, \{4,8\}, \{6,12\}, \{5,11\}\}$, \\
$\{\{1,10\}, \{2,9\}, \{3,8\}, \{4,7\}, \{5,12\}, \{6,11\}\}$,\\
$\{[8,1,11],[7,2,12],[4,5,9], [3,6,10]\},\{ [11,3,5],  [12,4,6], [1,7,9], [2,8,10]\}$,\\
$\{[6,9,8],[5,10,7],[2,11,4], [1,12,3]\}$.
\item
$(1,3,2)$: \\
$\{\{1,6\}, \{2,5\}, \{3,10\}, \{4,9\}, \{7,12\}, \{8,11\}\}$,\\
$\{[8,1,11],[7,2,12],[4,5,9], [3,6,10]\},\{ [11,3,5],  [12,4,6], [1,7,9], [2,8,10]\}$,\\
$\{[6,9,8],[5,10,7],[2,11,4], [1,12,3]\}$,\\
$\{[9,1,10,4],[2,6,11,5], [12,8,3,7]\},\{ [1,5,12,6],[3,9,2,10], [8,4,7,11]\}$.

\end{itemize}
\end{proof}

\begin{lemma}
\label{lemmaA2} URD-$(12;P_2, P_3, P_4)\supseteq D(12)$. %=$\{(11,0,0),(8,0,2), (5,0,4), (2,0,6),(7,$ $3,0), (4,3,2), (1,3,4), (3,6,0)\}$.
\end{lemma}
\begin{proof}
The case $((2,0,6)$ corresponds to a $(P_2, P_4)$-URD$(12;2,6)$  which is known to exist \cite{M1}.
The case $(0,6,2)$ corresponds to a $(P_3, P_4)$-URD$(12;6,2)$  which is known to exist \cite{GM1}. For all the other cases
take a $(K_{2}, P_3, P_4)$-URGDD$(r,s,t)$ of
type $6^2$ with $(r,s,t)\in \{(6,0,0), (3,0,2), (0,0,4), (2,3,0)\}$, which exists by
Lemma \ref{lemmaA1}. Fill in each  group of size 6 with
a copy of a  $(P_2, P_3, P_{4})$-URD$(6;r_1,s_1,t_1)$ with $(r_1,s_1,t_1)\in \{(5,0,0), (1,3,0)\}$, which exists by
Lemma \ref{lemmaP2}. This gives a $(K_{2},P_3, P_4)$-URD$(12;r,s,t)$ for every
$(r,s,t)\in \{(5,0,0), (1,3,0)\}+ \{(6,0,0), (3,0,2),$ $ (0,0,4), (2,3,0)\}$.
\end{proof}

\begin{lemma}
\label{lemmaD1} There exists a \/ $(P_2, P_3, P_{4})$-URGDD$(r,s,t)$ of type\/ $12^{2}$,
for every  $(r,s,t)\in \bar{D}(12^2)$.
%\{(12,0,0), (9,0,2), (6,0,4),(3,0,6), (0,0,8), (8,3,0), (5,3,2), (2,3,4),(4,6,0), (1,6,$ $ 2),(0,9,0)\}$.
\end{lemma}

\begin{proof}
The case $(0,9,0)$ corresponds to a $(P_3)$-URGDD$(9)$ of type  $12^{2}$ which is known to exist \cite{U}.
To obtain all remaining cases except $(1,6,2)$,
%=$\{(12,0,0), (9,0,2), (6,0,4),(3,0,6), (0,0,8), \\(8,3,0), (5,3,2), (2,3,4), (4,6,0), (1,6,2)\}$,
start from a 2-RGDD  of type $2^2$ with the block set partitioned into two 1-factors,  expand each point $6$ times and for each edge $e$ of a 1-factor  place on $e\times \{1,2,3,4,5,6\}$ a copy of a $(P_2, P_3, P_{4})$-URGDD$(r_1,s_1,t_1)$ of type $6^2$ (from
Lemma \ref{lemmaA1}) so to obtain a $(P_2, P_3, P_{4})$-URGDD$(r,s,t)$ of type\/ $12^{2}$, for every $(r,s,t)\in \bar{D}(12^2)\setminus \{(1,6,2),(0,9,0)\}$. For the case $(1,6,2)$ take  the groups $\{0,1,\ldots,11\}$ and $\{0',1',\ldots,11'\}$  and the
classes as listed below:\\
$\{\{0,11'\}, \{1,4'\}, \{2,6'\}, \{3,7'\}, \{4,1'\}, \{5,0'\}, \{6,10'\}, \{7,5'\}, \{8,3'\}, \{9,8'\}, \{10,2'\},\\\{11,9'\}\}$,\\
$\{[0',0,1'],[2',1,3'],[6',6,7'], [8',7,9'], [2,4',3],  [4,5',5], [8,10',9], [10,11',11]\}$,\\
$\{[1',2,5'],[2',5,4'],[7',8,11'], [8',11,10'], [1,0',3],  [7,6',9], [0,3',4], [6,9',10]\}$,\\
$\{[4',4,0'],[5',3,3'],[10',10,6'], [11',9,9'], [1,1',5],  [0,2',2], [7,7',11], [6,8',8]\}$,\\
$\{[0',6,1'],[2',7,3'],[6',0,7'], [8',1,9'], [2,10',3],  [4,11',5], [8,4',9], [10,5',11]\}$,\\
$\{[1',8,5'],[2',11,4'],[7',2,11'], [8',5,10'], [1,6',3],  [7,0',9], [0,9',4], [6,3',10]\}$,\\
$\{[4',10,0'],[5',9,3'],[10',4,6'], [11',3,9'], [1,7',5],  [0,8',2], [7,1',11], [6,2',8]\}$,\\
$\{[4',0,5',1],[6,11',7,10'], [0',2,3',5], [11,6',8,9'],[1',3,2',4], [9,7',10,8']\}$,\\
$\{[7,4',6,5'],[0,10',1,11'], [8,0',11,3'],[2,9',5,6'],[3,8',4,7'], [10,1',9,2']\}$.
\end{proof}

\begin{lemma}
\label{lemmaD2} There exists a \/ $(P_2,P_3, P_{4})$-URGDD$(1,6,10)$ of type\/ $12^{3}$.
\end{lemma}

\begin{proof}
Let $G_i=3Z_{36}+i$, $ i=0,1,2$. We
construct a $(P_2,P_3, P_{4})$-URGDD$(1,6,10)$ of type\/ $12^{3}$ by listing its  classes as follows:\\
$\{\{1+4i,2+4i\}, \{3+4i,8+4i\}: i=0,1\ldots,8, i\in Z_{36}\}$\\
$\{[3i,2+3i,4+3i]: i=0,1\ldots,11, i\in Z_{36}\}$,\\ $\{[1+3i,3+3i,11+3i]: i=0,1\ldots,11, i\in Z_{36}\}$,\\ $\{[3i,4+3i,8+3i]: i=0,1\ldots,11, i\in Z_{36}\}$,\\ $\{[6+3i,2+3i,10+3i]: i=0,1\ldots,11, i\in Z_{36}\}$,\\$\{[9+3i,1+3i,17+3i]: i=0,1\ldots,11, i\in Z_{36}\}$,\\ $\{[34+3i,18+3i,2+3i]: i=0,1\ldots,11, i\in Z_{36}\}$,\\
$\{[2+4i,3+4i,4+4i,5+4i]: i=0,1\ldots,8, i\in Z_{36}\}$,\\
$\{[4i,5+4i,10+4i,15+4i]: i=0,1\ldots,8, i\in Z_{36}\}$,\\
$\{[4i,7+4i,14+4i,1+4i]: i=0,1\ldots,8, i\in Z_{36}\}$,\\
$\{[11+4i,1+4i,8+4i,18+4i]: i=0,1\ldots,8, i\in Z_{36}\}$,\\
$\{[12+4i,2+4i,9+4i,23+4i]: i=0,1\ldots,8, i\in Z_{36}\}$,\\
$\{[14+4i,3+4i,13+4i,4i]: i=0,1\ldots,8, i\in Z_{36}\}$,\\
$\{[11+4i,4i,14+4i,25+4i]: i=0,1\ldots,8, i\in Z_{36}\}$,\\
$\{[12+4i,1+4i,18+4i,31+4i]: i=0,1\ldots,8, i\in Z_{36}\}$,\\
$\{[2+4i,19+4i,32+4i,13+4i]: i=0,1\ldots,8, i\in Z_{36}\}$,\\
$\{[17+4i,3+4i,20+4i,6+4i]: i=0,1\ldots,8, i\in Z_{36}\}$.

\end{proof}

\begin{lemma}
\label{lemmaD3} There exists a \/ $(P_2, P_3, P_{4})$-URGDD$(r,s,t)$ of type\/ $12^{3}$, for every 
$(r,s,t)\in \bar{D}(12^{3})$.
\end{lemma}

\begin{proof} 
The case $(0,9,8)$ corresponds to a $(P_2, P_3,P_4)$-URGDD$(0,9,8)$ of type  $12^{3}$ which is known to exist \cite{GM1}, while the case
$(1,6,10)$  is given by Lemma \ref{lemmaD2}.\\
For the cases $(4,15,0), (0,18,0), (5,12,2),(1,15,2), (2,12,4), (3,9,6)$,
start from a 3-RGDD $\cD$ of type $3^3$ with three parallel classes, expand each point $4$ times and for each block $b$ of a given parallel class of $\cD$ place on $b\times \{1,2,3,4\}$ a copy of a  $(P_2, P_3, P_{4})$-URGDD$(r_1,s_1,t_1)$ of type\/ $4^{3}$, with $(r_1,s_1,t_1)\in \{(0,6,0), (4,3,0),(1,3,$ $2)\}$, which exists by Lemma \ref{lemmaA11}. Since $\cD$ contains three parallel classes the result is a $(P_2, P_3,$ $ P_{4})$-URGDD$(r,s,t)$ of type\/ $12^{3}$, for every $(r,s,t)\in 3*\{(0,6,0), (4,3,0),(1,3,$ $2)\}\supseteq\{(4,15,0), (0,18,0), (5,12,2),(1,15,2), (2,12,4), (3,9,6)\}$.\\
To settle the remaining cases, start from a 2-RGDD of type $2^3$ with the block set partitioned into four 1-factors, expand each point $6$ times and for each edge $e$ of a 1-factor  place on $e\times \{1,2,3,4,5,6\}$ a copy of a $(P_2, P_3, P_{4})$-URGDD$(r_1,s_1,t_1)$ of type\/ $6^{2}$, with $(r_1,s_1,t_1)\in \{(6,0,0), (3,0,$ $2), (0,0,4), (2,3,0)\}$, which exists by Lemma \ref{lemmaA1}. The result is a $(P_2, P_3, P_{4})$-URGDD$(r,s,t)$ of type\/ $12^{3}$,
for every $(r,s,t)\in 4* \{(6,0,0), (3,0,2), (0,0,4), (2,3,0)\} \,\supseteq\, \{\bar{D}(12^{3})\setminus\{(4,15,0), (0,18,0), (5,12,2),(1,15,$ $2), (2,12,4), (3,9,6),(0,9,8),(1,6,10)\}$.
\end{proof}

\begin{lemma}
\label{lemmaD36}
URD-$(36;P_2, P_3, P_4)\supseteq D(36)$.
\end{lemma}
\begin{proof}
Start from a $3$-RGDD of type $1^{3}$  and apply Theorem \ref{thmR2} with $t=12, g=1, u=3$
(the input designs are a $(P_2, P_3,P_4)$-URGDD$(r_1, s_1,t_1)$ of type $12^{3}$  with $(r_1,s_1,t_1)\in \bar{D}(12^{3})$, which exists by  Lemma \ref{lemmaD3}, and a $(P_2, P_3,P_4)$-URD$(12;r_2,s_2,t_2)$ with $(r_2,s_2,t_2)\in D(12)$, which exists by Lemma \ref{lemmaA2}). This implies
$$ URD(v; P_2, P_3, P_4)\supseteq D(12)+\bar{D}(12^{3}).$$
Taking into account that
$$D(12)= D_0(12)\cup D_1(12)\cup D_2(12)$$
and
$$\bar{D}(12^{3})=\bar{D}_0(12^{3})\cup \bar{D}_1(12^{3})\cup \ldots \cup \bar{D}_{6}(12^{3}).$$
Since it is easy to see that, for each $j=0,1,\ldots,6$
$$ \bar{D}_{j}(12^{3})+D_2(12)=D _{j+2}(36) $$
and, for each $i=0,1$,
$$ \bar{D}_{0}(12^{3})+D_{i}(12)=D _{i}(36), $$
we obtain
$D(12)+\bar{D}(12^{3})=D(36)$ and
this complete the proof.
\end{proof}

\begin{lemma}
\label{lemmaD5} There exists a \/ $(P_2, P_3,P_4)$-IURD$(36,12;[r_1,s_1,t_1],[r_2,s_2,t_2])$ for e\-ve\-ry
 $(r_1,s_1,t_1)\in D(12)$ and for every $(r_2,s_2,t_2) \in 2 \ast \bar{D}(12^{2})$.
\end{lemma}
\begin{proof}
Lemma \ref{lemmaF1} gives $2*\bar{D}(12^{2})=\bar{D}(12^{3})$. Start from a $(P_2, P_3, P_{4})$-URGDD$(r_2,$ $s_2,t_2)$ of type\/ $12^{3}$  with
$(r_2,s_2,t_2)\in \bar{D}(12^{3})$, which exists by Lemma \ref{lemmaD3}, and fill in two groups of size 12 with a copy of a
$(P_2, P_3,P_4)$-URD$(12;r_1,s_1,t_1)$ with $(r_1,s_1,t_1)\in D(12)$, which exists by Lemma \ref{lemmaA2}.
\end{proof}

In order to obtain our main result, we need to handle a further case, $v=60$, which will be discussed in a separate section.

\subsection{The case $v=60$}

\begin{lemma}
\label{lemmaA4}{\rm \cite{M1}} URD-$(20;P_2, P_3, P_4)=D(20)=\{(19-3x,0,2x),x=0,1,\ldots,6\}$.
\end{lemma}

In what follows, we will denote by $C_{m(n)}$ the graph $G_{(n)}$ where $G$ is a $m$-cycle (a $m$-cycle $C_m$ with vertex set $\{a_1,a_2,\ldots, a_m\}$ and edge set 
$\{\{a_1,a_2\}, \{a_2,a_3\}, \ldots, \{a_{m-1},$ $a_m\}, \{a_{m},a_1\}\}$
 will be denoted by 
$(a_1,a_2,\ldots, a_m)$).

\begin{lemma}
\label{lemmaD601} There exists a resolvable $P_3$-decomposition of $C_{5(6)}$.

\end{lemma}

\begin{proof}
On $Z_6\times Z_5$ consider the set $\cB$ of copies of $P_3$ obtained by developing in $Z_6$ the following base blocks (partitioned into three sets, for convenience):

\bigskip

$\cB_0^{(1)}$: \ $[0_1,0_0,1_1], [2_2,2_1,3_2], [1_3,1_2,2_3], [2_4,0_3,3_4], [1_0,1_4,2_0]$;

\bigskip

$\cB_0^{(2)}$: \ $[2_1,0_0,3_1], [3_2,1_1,4_2], [4_3,2_2,5_3], [4_4,0_3,5_4], [4_0,0_4,5_0]$;

\bigskip

$\cB_0^{(3)}$: \ $[4_1,0_0,5_1], [4_2,0_1,5_2], [4_3,0_2,5_3], [0_4,0_3,1_4], [4_0,2_4,5_0]$.

\bigskip

\noindent For $j=1,2,3$ and $i\in Z_6$, let $\cB_i^{(j)}=\{b+i: b\in \cB_0^{(j)}\}$. The blocks of $\cB$ can be partitioned into the nine parallel classes $\cB_i^{(j)}\cup \cB_{3i}^{(j)}$, for $j=1,2,3$ and $i=0,1,2$.
\end{proof}

\begin{lemma}
\label{lemmaE2} There exists a \/ $(P_2, C_{5})$-URGDD$(s,t)$ of type\/ $2^{5}$, for every 
$(s,t)\in \{(8,0),(6,1),(4,2)\}$.
\end{lemma}
\begin{proof} The case $(8,0)$ is trivial. \\
Take $Z_{10}$ as vertex set and $G_i=5Z_{10}+i$, $ i=0,1,2,3,4$ as groups;  the  classes listed below:
\begin{itemize}
\item
$(6,1)$: \\
$\{\{2i,1+2i\}: i=1,2,4, i\in Z_{10}\}\cup \{\{0,6\},\{1,7\}\}$,\\
$\{\{1+2i,2+2i\}: i=0,1,2,3,4, i\in Z_{10}\}$,\\
$\{\{2i,3+2i\}: i=0,1,2,3,4, i\in Z_{10}\}$,\\
$\{\{1+2i,4+2i\}: i=0,1,2,3,4, i\in Z_{10}\}$,\\
$\{\{0,4\},\{8,2\},\{1,5\},\{9,3\},\{6,7\}\}$,\\
$\{\{4,8\},\{2,6\},\{5,9\},\{3,7\},\{0,1\}\}$,\\
$\{(i,2+i,4+i,6+1,8+i): i=0,1, i\in Z_{10}\}$.
\item
$(4,2)$:\\
$\{\{2i,1+2i\}: i=0,1,2,3,4, i\in Z_{10}\}$,\\
$\{\{1+2i,2+2i\}: i=0,1,2,3,4, i\in Z_{10}\}$,\\
$\{\{2i,3+2i\}: i=0,1,2,3,4, i\in Z_{10}\}$,\\
$\{\{1+2i,4+2i\}: i=0,1,2,3,4, i\in Z_{10}\}$,\\
$\{(i,2+i,4+i,6+i,8+i): i=0,1, i\in Z_{10}\}$\\
$\{(i,4+i,8+i,2+i,6+i): i=0,1, i\in Z_{10}\}$.
\end{itemize}
\end{proof}

Define the following set of triples:\\
$I_1 =\{(48-3x,0,2x), x=0,1,,\ldots, 16\}$, $I_2 =\{(44-3x,3,2x), x=0,1,,\ldots, 14\}$,\\
$I_3 =\{(36-3x,9,2x), x=0,1,,\ldots, 12\}$, $I_4 =\{(32-3x,12,2x), x=0,1,,\ldots, 10\}$,\\
$I_5 =\{(24-3x,18,2x), x=0,1,,\ldots, 8\}$, $I_6 =\{(20-3x,21,2x), x=0,1,,\ldots, 6\}$,\\
$I_7 =\{(0,36,0)\}$, $I_8 =\{(4,33,0), (1,33,2)\}$.

\begin{lemma}
\label{lemmaD2bis} There exists a \/ $(P_2, P_3, P_{4})$-URGDD$(r,s,t)$ of type $12^5$,
for every $(r,s,t)\in \cup_{j=1}^{8}I_{j}$.
\end{lemma}

\begin{proof}
Start from a $(P_2, C_{5})$-URGDD$(s,t)$ of type\/ $2^{5}$ \cite{RS}, with
$(s,t)\in \{(8,0),(6,1),$ $(4,2)\}$, which exists by Lemma \ref{lemmaE2}, expand each point $6$ times and for every block $b$ of a parallel class place on $b\times \{1,2,3,4,5,6\}$ a copy of a 
 \/ $(P_2, P_3, P_{4})$-URGDD$(r,s,t)$ of type\/ $6^{2}$
with $(r_1,s_1,t_1)\in \{(6,0,0), (3,0,2), (0,0,4), (2,3,0)\}$, from Lemma \ref{lemmaA1}, or, as the case may be, a copy of a resolvable $P_3$-decomposition of $C_{5(6)}$ from Lemma \ref{lemmaD601}. This gives a \/ $(P_2, P_3, P_{4})$-URGDD$(r,s,t)$ of type $12^5$ with
$(r,s,t)\in A_1\cup A_2\cup A_3$, where  $ A_1=8\ast\{(6,0,0), (3,0,2), (0,0,4), (2,3,0)\}$, $ A_2=6\ast\{(6,0,0), (3,0,2), (0,0,4),$ $ (2,3,0)\}+(0,9,0)$, $ A_3=4\ast\{(6,0,0), (3,0,2), (0,0,4), (2,3,0)\}+(0,18,0)$, and since  $I_i\cup I_{i+1}\subseteq A_{\frac{i+1}{2}}$, $i=1,3,5$, $(r,s,t)\in \cup_{j=1}^{6}I_{j}$.\\
To obtain the remaining triples in $I_7\cup I_8$,  start from  a 3-RGDD of type $3^{5}$, expand each point  $4$ times and
for every block $b$ of a given parallel class place on $b\times \{1,2,3,4\}$ a copy of a  $(P_2, P_3, P_{4})$-URGDD$(r_1,s_1,t_1)$ of type\/ $4^{3}$,
with $(r_1,s_1,t_1)\in \{(0,6,0), (4,3,0),(1,3,2)\}$,  which exists by Lemma \ref{lemmaA11}. The result is a \/ $(P_2, P_3, P_{4})$-URGDD$(r,s,t)$ of type $12^5$ with  $(r,s,t)\in 6\ast\{(0,6,0), (4,3,0),(1,3,2)\}\supseteq I_i$, for $i=7,8,$ and this completes the proof.\end{proof}

%Consider now the set of triples: 
%$I_7 =\{(0,36,0)\}$, $I_8 =\{(4,33,0), (1,33,2)\}$, $I_9=\{(0,30,0)\}$.

\begin{lemma}
\label{lemmaD4} There exists a \/ $(P_2, P_3, P_{4})$-URGDD$(0,30,0)$ of type $20^{3}$.
\end{lemma}

\begin{proof}
Start from  a 3-RGDD of type $5^{3}$, expand every point $4$ times and for every block $b$ of a given parallel class place on $b\times \{1,2,3,4\}$ a copy of a  $(P_2, P_3, P_{4})$-URGDD$(r_1,s_1,t_1)$ of type\/ $4^{3}$,
with $(r_1,s_1,t_1)\in \{(0,6,0), (4,3,0),(1,3,2)\}$,  which exists by Lemma \ref{lemmaA2}.
 The result is a \/ $(P_2, P_3, P_{4})$-URGDD$(r,s,t)$ of type $20^{3}$ with  $(r,s,t)\in 5\ast\{(0,6,0), (4,3,0),(1,3,2)\}\supseteq \{(0,30,0)\}$.
\end{proof}

\begin{lemma}
\label{lemmaD60} URD-$(60;P_2, P_3, P_4)\supseteq D(60)$.
\end{lemma}
\begin{proof}
The case $(0,33,10)$ corresponds to a $(P_2, P_3, P_4)$-URD$(60;0,33,10)$  which is known to exist \cite{GM1}.
 For $D_{10}(60)$, start from  a \/ $(P_2, P_3, P_{4})$-URGDD$(0,30,0)$ of type $20^{3}$, which exists by Lemma  \ref{lemmaD4}, and fill in each group with a copy of a $(P_2, P_3, P_4)$-URD$(20;r_2,0,t_2)$ from Lemma \ref{lemmaA4} so to obtain the set of triples
$I_9+D(20)$=$D_{10}(60)$.

For all the other cases, start from  a \/ $(P_2, P_3, P_{4})$-URGDD$(r_1,s_1,t_1)$ of type $12^{5}$ from Lemma  \ref{lemmaD2bis} and fill in each group with a copy of a $(P_2, P_3, P_4)$-URD$(12;r_2,s_2,t_2)$ from Lemma \ref{lemmaA2}
so to obtain the following triples:\\

$I_1+D_i(12)$=$D_i(60)$, $i=0,1,2$, 

$I_2+D_2(12)$= $D_3(60)$, $I_3+D_1(12)$= $D_4(60)$, $I_3+D_2(12)$= $D_5(60)$,

$I_5+D_i(12)$=$D_{i+6}(60)$, $i=0,1,2$, 

$I_6+D_2(12)$= $D_9(60)$, $I_8+D_0(12)$= $D_{11}-\{(0,33,10)\}$,

$I_7+D_i(12)$=$D_{i+12}(60)$, $i=0,1,2$.
\end{proof}

\section{The case $v\equiv 0\pmod{24}$}

\begin{lemma}
\label{lemmaC2} For every \ $v\equiv 0\pmod{24}$, $
URD(v;P_2,P_3,P_4) \supseteq D(v)$.
\end{lemma}

\begin{proof}
Start from a $2$-RGDD of type $1^{\frac{v}{12}}$  \cite{CD} and apply Theorem \ref{thmR2} with $t=12, g=1, u= \frac{v}{12}$
(the input designs are a $(P_2, P_3,P_4)$-URGDD$(r_1, s_1,t_1)$ of type $12^{2}$  with $(r_1,s_1,t_1)\in \bar{D}(12^2)$, which exists by  Lemma \ref{lemmaD1}, and a $(P_2, P_3,P_4)$-URD$(12;r_2,s_2,t_2)$ with $(r_2,s_2,t_2)\in D(12)$, which exists by  Lemma \ref{lemmaA2}). This implies
$$ URD(v; P_2, P_3, P_4)\supseteq D(12)+\frac{v-12}{12}\ast D(12^{2})$$.
Since 
$$D(12)= D_0(12)\cup D_1(12)\cup D_2(12)$$
 and
$$\frac{v-12}{12}\ast  \bar{D}(12^{2})= \bar{D}((v-12)^{2})= \bar{D}_0((v-12)^{2})\cup  \bar{D}_1((v-12)^{2})\cup \ldots \cup  \bar{D}_{\frac{v-12}{4}}((v-12)^{2}),$$
and, as it is easy to see, for each $j=0,1,\ldots,\frac{v-12}{4}$
$$  \bar{D}_{j}((v-12)^{2})+D_2(12)=D _{j+2}(v) $$
and, for each $i=0,1$,
$$  \bar{D}_{0}((v-12)^{2})+D_{i}(12)=D _{i}(v), $$
we obtain
$$D(12)+\frac{v-12}{12}\ast D(12^{2})=\bigcup_{j=0}^{\frac{v-4}{4}}D _{j}(v)= D(v).$$
\end{proof}

\section{The case $v\equiv 12\pmod{24}$}

\begin{lemma}
\label{lemmaF2} For every \ $v\equiv 12\pmod{24}$, $D(v)\subseteq
URD(v;P_2,P_3,P_4)$.
\end{lemma}

\begin{proof}
The cases $v=12,36,60$ are covered by Lemmas\ref{lemmaA2}, \ref{lemmaD36} and \ref{lemmaD60}.  For $v>60$  start from a $2$-frame of type $6^{\frac{v-12}{24}}$ (\cite{S2}) and apply Theorem \ref{thmR3} with $g=6, u= \frac{v-12}{24}, t=12$ and $h=12$
(the input designs are: a $(P_2, P_3,P_4)$-URD$(12;r_1,s_1,t_1)$ with $(r_1,s_1,t_1)\in D(12)$, which exists by  Lemma \ref{lemmaA2};  a $(P_2, P_3,P_4)$-URGDD$(r_2,s_2,t_2)$ of type $12^2$ with $(r_2,s_2,t_2)\in\bar{D}(12^2)$, which exists by  Lemma \ref{lemmaD1}; a   $(P_2, P_3,P_4)$-IURD$(36,12;[r_1,s_1,t_1],[r_3,s_3,t_3])$ with
 $(r_1,s_1,t_1)\in D(12)$ and $(r_3,s_3,t_3) \in 2\ast \bar{D}(12^{2})$, which exists by  Lemma \ref{lemmaD5}). This implies
$$ URD(v; P_2, P_3, P_4)\supseteq D(12)+ \frac{v-12}{24}\ast\left(2\ast \bar{D}(12^{2})\right).$$
Since $$D(12)= D_0(12)\cup D_1(12)\cup D_2(12)$$
 and

\bigskip

$\frac{v-12}{24}\ast\left( 2\ast \bar{D}(12^{2}) \right)=$ 

\bigskip

$=\bar{D}\left((v-12)^{2}\right)\cup \bar{D}_0\left((v-12)^{2}\right)\cup \bar{D}_1\left((v-12)^{2}\right)\cup \ldots \cup \bar{D}_{\frac{v-12}{4}}\left((v-12)^{2}\right)$.

\bigskip

\noindent and, as it is easy to see, for each $j=0,1,\ldots,\frac{v-12}{4}$
$$ \bar{D}_{j}\left((v-12)^2\right)+D_2(12)=D _{j+2}(v) $$
and, for each $i=0,1$,
$$ \bar{D}_{0}\left((v-12)^2\right)+D_{i}(12)=D _{i}(v), $$
we obtain
$$D(12)+\frac{v-12}{24}\ast\left( 2\ast \bar{D}(12^{2} \right)=\bigcup_{j=0}^{\frac{v-4}{4}}D _{j}(v)= D(v).$$
\end{proof}

\section{Conclusion}

We are now in a position to prove the following main result.

\begin{thm}
For every \ $v\equiv 0\pmod{12}$,
% $v\neq6,12$,
\/ $URD(v;P_2, P_3, P_4)=D(v)$.
\end{thm}
\begin{proof}
Necessity follows by Lemmas \ref{lemmaP0}. %, \ref{lemmaP1} ,\ref{lemmaP2} and \ref{lemmaP3}.
Sufficiency follows by Lemmas  \ref{lemmaF2} and
\ref{lemmaC2}.
\end{proof}


\begin{thebibliography}{9}



\bibitem{BHY} J. C. Bermond, K. Heinrich and M.L. Yu, {\em Existence of resolvable path designs}, Europ. J. Combinatorics {\bf 11}
(1990), 205--211.


\bibitem {B} C. Berge, {\em Graphs and Hypergraphs}, Elsevier, 1973.

\bibitem{CD} C. J. Colbourn and J. H. Dinitz (eds.), {\em  Handbook of Combinatorial Designs}, Second Edition, Chapman and Hall/CRC, Boca Raton, FL, 2007.

\bibitem{URD} P. Danziger and E. Mendelsohn, {\em Uniformly Resolvable Designs}, J.
Combin. Math. Combin. Comput. {\bf 21} (1996), 65--83.

\bibitem{DQS}P. Danziger, G. Quattrocchi and B. Stevens, 2004. {\em The Hamilton-Waterloo Problem for Cycle Sizes $3$ and $4$}, J. Comb. Des. 12,
221--232.

\bibitem{DLD} J. H. Dinitz, A. C. H. Ling and P. Danziger,
 {\em Maximum Uniformly resolvable designs with  block sizes $2$ and $4$},
      Discrete Math. {\bf 309} (2009), 4716--4721.

\bibitem{GM} M. Gionfriddo and S. Milici, {\em On the existence of uniformly resolvable decompositions of\/ $K_v$ and\/ $K_v-I$ into paths and kites},
Discrete Math. {\bf 313} (2013), 2830--2834.

\bibitem{GM1} M. Gionfriddo and S. Milici, {\em Uniformly resolvable $\cH$-designs with $\cH$=$\{P_3, P_4\}$}, manuscript $(2014)$.


\bibitem{HR} P. Horak, R. Nedela and A. Rosa, {\em The Hamilton-Waterloo problem:
the case of Hamilton factors and triangle-factors}, Discrete Math.
{\bf 284} (2004) 181--188.

\bibitem{H} J. D. Horton, {\em Resolvable path designs}, J. Comb. Th. Ser. A {\bf 39} (1985) 117-131.

\bibitem{KMT} S. Kucukcifci, S. Milici and Zs. Tuza,
{\em Maximum uniformly resolvable decompositions of\/ $K_v$ into\/ $3$-stars and\/ $3$-cycles},
 to appear in Discrete Math..
 
\bibitem{KLMT} S. Kucukcifci, G. Lo Faro, S. Milici and A. Tripodi,
{\em Resolvable\/ $3$-star designs}, manuscript (2014).

\bibitem{M} S. Milici, {\em A note on uniformly resolvable decompositions of\/ $K_v$ and\/ $K_v-I$ into\/ $2$-stars and\/ $4$-cycles},  Austalas. J. Combin.,{\bf 56} (2013), 195--200.

\bibitem{M1} S. Milici, {\em Uniformly resolvable $\{K_2, P_k\}$-designs with $k$=$\{3,4\}$}, manuscript $(2014)$.

\bibitem{MT} S. Milici and Zs. Tuza,
 {\em Uniformly resolvable decompositions of\/ $K_v$ into\/ $P_3$ and\/ $K_3$ graphs}, Discrete Math.
{\bf 331} (2014) 137--141.


\bibitem{R} R. Rees, {\em Uniformly resolvable pairwise balanced
designs with block sizes two and three}, J. Comb. Th. Series A
{\bf 45} (1987) 207-225.

\bibitem{RS}  R. Rees and D. R. Stinson.
{\em On resolvable group divisible designs with block size $3$, Ars
Combin.} {\bf 23} (1987), 107-120.

\bibitem{S1}
 E. Schuster, {\em Uniformly resolvable designs with index one and block sizes
  three and four -- with three or five parallel classes of block size four},
 Discrete Math. {\bf 309} (2009), 2452--2465.
 
 \bibitem{S1bis}
 E. Schuster, {\em Small uniformly resolvable designs with block sizes
  $3$ and $4$},
J. Comb. Des.  {\bf 21} (2013), 481--523.

\bibitem{S2}
 E. Schuster, {\em Uniformly resolvable designs with index one, block sizes
 three and five and up to five parallel classes with blocks of size five},
Discrete Math. {\bf 309} (2009), 4435--4442.

\bibitem{SG}
 E. Schuster and G. Ge, {\em On uniformly resolvable designs with block sizes
  $3$ and $4$},
Des. Codes Cryptogr.  {\bf 57} (2010), 57--69.

\bibitem{U}  K. Ushio, {\em $P_3$-factorization of complete bipartite graphs}, Discrete Math. {\bf 72} (1988),
    361-366.

\bibitem{UT}  K. Ushio and R. Tsuruno, {\em $P_3$-factorization of complete multipartite graphs}, Graphs and Combinatorics {\bf 5} (1989),
    385-387.
    
    
\bibitem{Y} M.L. Yu, {\em Resolvable path designs of complete graphs}, Thesis B. Sc. Fudan University 1983, Ext. Dep. Simon Fraser University 1985.


\end{thebibliography}
\end{document}